\documentclass[final]{siamltex}
\newtheorem{example}[theorem]{Example}

\newtheorem{remark}[theorem]{Remark}

\newcommand{\R}{\rm I\kern-.19emR}
\newcommand{\bQ}{\rm I\kern-.19emQ}
\newcommand{\C}{\rm I\kern-.17emC}
\newcommand{\bR}  {\R}
\newcommand{\bRn}  {\R^{n,n}}
\newcommand{\bRmm} {\R^{m,m}}
\newcommand{\bRnmnm} {\R^{n-m,n-m}}

\newcommand{\aij} {a_{i,j}}
\newcommand{\aji} {a_{j,i}}

\newcommand{\beqo} {\begin {eqnarray*}}
\newcommand{\beq} {\begin {eqnarray}}
\newcommand{\eeqo} {\end {eqnarray*}}
\newcommand{\eeq} {\end {eqnarray}}

\newcommand{\bdefi} {\begin {definition}}
\newcommand{\edefi} {\end {definition}}
\newcommand{\bpro} {\begin {proposition}}
\newcommand{\epro} {\end {proposition}}
\newcommand{\btheo} {\begin {theorem}}
\newcommand{\etheo} {\end {theorem}}
\newcommand{\blem} {\begin {lemma}}
\newcommand{\elem} {\end {lemma}}
\newcommand{\bcor} {\begin {corollary}}
\newcommand{\ecor} {\end {corollary}}
\newcommand{\zmas} {$Z$--matrices}
\newcommand{\mmas} {$M$--matrices}

\newcommand{\tom} {\delta (A)}

\newcommand{\gum} {generalized ultrametric matrix}
\newcommand{\stum} {strictly  ultrametric matrix}
\newcommand{\gums} {generalized ultrametric matrices}

\newcommand{\fmas} {$F_0$--matrices}
\newcommand{\mma} {$M$--matrix}

\newcommand{\fma} {$F_0$--matrix}

\newcommand{\nhalf}{\lfloor\frac{n}{2}\rfloor}
\newcommand{\sgum} {shifted generalized ultrametric matrix}
\newcommand{\sgums} {shifted generalized ultrametric matrices}
\title {Inverse  acyclic Z-matrices, matrices of tree structure and shifted generalized ultrametric matrices }

\author{Reinhard Nabben\thanks{
     	TU Berlin 
        Germany, email: nabben@math.tu-berlin.de, date: 27.06.2026}}

\date{27.06.2026}

\begin{document}

\maketitle

\begin{abstract}
Inverses of acyclic or treediagonal matrices are nicely characterized in the landmarked paper by Klein \cite{KLe82}.  But also in \cite{Nab01} a   characterization of these inverses is given which is related to  the structure of inverses of tridiagonal matrices. Here  we use this characterization to determine and construct matrices whose inverses are acyclic $Z$-matrices. Moreover, we establish relations between several classes of matrices, namely the classes of $Z$-matrices, inverse $Z$-matrices, ultrametric matrices, generalized ultrametric matrices, shifted ultrametric matrices, acyclic matrices, inverse acyclic matrices, matrices of tree structure, cyclopes and generalized cyclopes. 
\end{abstract}



\pagestyle{myheadings}
\thispagestyle{plain}

\markboth{R. NABBEN}{inverse acyclic Z-matrices}

\section{Introduction}

An 
$n \times n$ matrix is called {\it data sparse} if it can be represented by much less than  $n^2$ parameters with respect to a certain format or structure.
The most prominent class of  data sparse matrices is the class of inverses of non-singular tridiagonal matrices, which can be described by $3n-2$ parameters or $2n-1$ parameter in the symmetric case \cite{GK1,GK2}. But  also inverses of nonsingular irreducible acyclic matrices are data  sparse and  can be described also by  $3n-2$ parameters \cite{Nab01}. 
An irreducible  matrix is called acyclic or treediagonal if its undirected graph is a tree. Note  that the graph  of a tridiagonal matrix is just a line. Hence the class of acyclic matrices includes  the class  of tridiagonal matrices. 

Acyclic (or treediagonal) matrices are studied in the landmarked paper by Klein \cite{KLe82}. He showed that the inverse of tridiagonal matrices satisfies  the so-called triangle property. In \cite{Nab01} the author  gives   another equivalent formulation  for  the inverses of these matrices. 
The beauty of this result is that it nicely describes the structure in a way it is done  for  inverses of tridiagonal matrices. It is proven  
that the inverses of irreducible acyclic symmetric matrices are given 
as the Hadamard product of three matrices, a type D matrix, a flipped type
D matrix  and a matrix of tree structure which is closely related 
to the graph of $A$ itself. A similar result holds for non-symmetric matrices.

Acylic matrices are  studied by many authors, however often eigenvalue problems or   combinatorial properties of these matrices are in the main focus, see e.g. \cite{ Bru86, DemG93, Fied75, Kim09}.  But in \cite{Bri04}  Moore--Penrose inverses of acyclic matrices are considered. 

On the other hand  there is an on going interest in $Z$-matrices and matrices whose inverse is a $Z$-matrix (inverse $Z$-matrices). Fiedler and Markam introduced in \cite{FieM92} a classification of $Z$-matrices. The best known class  of $Z$-matrices is the class of $M$-matrices. 
There exist many equivalent characterizations of $M$-matrices \cite{BerP94}. Inverse $M$-matrices are also studied  by many authors, see
\cite{Joh82, JohS11} for an overview. But also for other classes of $Z$-matrices, namely the  classes of $N_0$- and $F_0$-matrices, characterizations are known \cite{JohG82, Joh85, Che90}. 
In \cite{Nab97} the author established properties and characterization 
for all classes of $Z$-matrices and inverse $Z$-matrices. 
Inverse $Z$-matrices occur e.g. in the theory of linear comnplementarity problems, where such a problem can be reduced to a single linear  problem, if the relevant matrix is an inverse $Z$-matrix. Moreover, semimonton matrices are  inverse $Z$-matrices \cite{TsaW19}.

In \cite{McDNST95} and \cite{NabV95a} the  class of generalized ultrametric matrices is  introduced, which includes the class of symmetric ultrametric matrices \cite{NabV94}. It is then proved in \cite{McDNST95} and \cite{NabV95a}  that the inverses of these matrices  are $M$-matrices. Motivated by this result in \cite{NabV95b} 
shifted generalized ultrametric matrices are  defined. It is then established that  with certain shifts one can obtain inverse $Z$-matrices of all classes of $Z$-matrices.
Moreover, in \cite{Nab97} the so-called type D-matrices (introduced by Markham in \cite{Mar72} as inverse $M$-matrices) are modified such that one can obtain inverse $Z$-matrices easily.  Even more, the inverses of these type D matrices are tridiagonal. 
In \cite{McDNNST98} so-called cyclopes are introduced which can be inverse tridiagonal $Z$-matrices of all kind. 
In \cite{MarMZ03} generalized cyclopes are defined which are then inverses of acyclic $M$-matrices.   

Inverses  of tridiagonal matrices and inverses of acyclic matrices share many properties. It is well-known  that  the absolute values of the entries of the inverse of diagonally dominant tridiagonal matrix decay away from the diagonal see e.g. \cite{Nab99b, Nab99a}. Inverses of  diagonally dominant acyclic matrices do not have  such a  decay. But there is still a decay, however, this decay is somehow hidden \cite{Nab26}.

Recently in \cite{PraS26} conditions for matrices  are  given such the inverses are treediagonal (acyclic) $M$-, $N_0-$ and $F_0$-matrices. However, these conditions do not lead directly to a construction of the related  matrices. Motivated by the results of \cite{PraS26} we give simple
ways to build  inverse acyclic $Z$-matrices of all  kind. 
We characterize not only acylic inverse $M$- and $N_0$-matrices but also give a simple construction of all classes of inverse acyclic $Z$-matrices  for an arbitrary given tree. 
Our results then generalize  some of the results of \cite{PraS26}. Even more, our results include general type-D matrices  and  so-called  cyclopes \cite{McDNNST98}, which lead to inverse tridiagonal $Z$-matrices.  
Based on results of \cite{Nab01} for inverse acyclic matrices,  in \cite{MarMZ03} inverse acyclic $M$-matrices are characterized by introducing generalized cyclopes. 
Here we generalize the main result of  \cite{MarMZ03}
for inverse $M$-matrices to all classes of inverse $Z$-matrices. 
Even more, our proofs are much shorter than the proof given in \cite{MarMZ03} for the $M$-matrix case.
Thus, the study of inverse acyclic matrices lead to many relations between  several classes of matrices, namely 
$Z$-matrices, inverse $Z$-matrices, ultrametric matrices, generalized ultrametric matrices, shifted ultrametric matrices, acyclic matrices, inverse acyclic matrices, matrices of tree structure, cyclopes and generalized cyclopes.

\section{Preliminary results -  some classes of matrices}

\subsection{Notation}
A real squared matrix $A$ is called nonnegative, denoted by $A \geq 0$, if all its entries are nonnegative. Similarly a positive (negative)  matrix $A$ is defined and denoted by $A > 0$ ($A < 0$). Moreover, let  $\rho(C)$ be the spectral radius of a square matrix $C$. The vector of all ones in $\bR^n$ is denoted by $\xi_n$, thus
\[  \xi_n := [1, \ldots, 1]^T \in \bR^n.\]
A matrix of all ones is denoted by $E$.  The set of positive integers
$\{1,\ldots,n\}$
 is denoted by $N$. For an integer $t$ we define the vectors 
\[ \tilde e_t := [0, \ldots, 0, 1]^T   \in \bR^t  \quad \mbox{and} \quad   \hat  e_t = [1, 0, \ldots, 0]^T \in \bR^t   . \]

For a block-partitioned  matrix $A \in \bRn$ with 
\beqo
A = \left[ \begin{array}{cc}
A_{1,1} & A_{1,2} \\
A_{2,1} & A_{2,2} 
\end{array} \right],
\eeqo
where $A_{1,1} \in \bR ^{r,r}$ is nonsingular and where $A_{2,2} \in \bR^{n-r,n-r}$ with $1 \leq r < n$, then
\beq \label{defschur}
A/A_{1,1} := A_{2,2} - A_{2,1}A^{-1}_{1,1}A_{1,2}
\eeq
denotes the {\it Schur complement} of $A$ with respect to $A_{1,1}$. Similarly, 
$A/A_{2,2}$ is defined. Also, $|M|$ denotes the cardinality of a set $M$, i.e., 
the number of elements in $M$.

 \subsection{Z-matrices and inverse Z-matrices}
A  real matrix whose off-diagonal entries are nonpositive is called $Z$-matrix. Every $Z$ matrix can be written as 
\beqo
A = tI - B, \ \mbox{for} \ t \in \bR \ \mbox{and} \ B \geq 0.
\eeqo

In 1992 Fiedler and Markham \cite{FieM92} introduced the following classification of \zmas:

\begin{definition} \label{defls}
Let $L_s$ (for $s = 0,\ldots,n$) denotes the class of matrices consisting of real $n \times n$ matrices which have the form
\beq \label{defzeq}
& A = tI - B \quad \mbox { where } &  \\
& B \geq 0 \quad \mbox{ and } \quad \rho_s(B) \leq t < \rho_{s+1}(B).
& \nonumber 
\eeq
Here, 
\beqo
\rho_s(B) := max \{ \rho (\tilde B) :  \tilde B \  \mbox{ is a } \ s \times s 
\ \mbox{principal submatrices of} \  B \},
\eeqo
 and we set $ \rho_{0}(B) := -\infty$ and $\rho_{n+1}(B) := \infty.$
\end{definition}


The class $L_n$ is just the class of $n \times n $ 
(singular and nonsingular) 
M--matrices. The class  $L_{n-1}$ is the class of $n \times n $ $N_0$--matrices introduced by G. 
Johnson \cite{JohG82}, and this class contains the N--matrices defined by K. Fan \cite{Fan66}.
Moreover, the class of $n \times n$  \fmas\ introduced by Johnson \cite{JohG82} is just $L_{n-2}$. 

Here we should mention that the classification of \zmas\ given above inherits the dimension of the matrices one consider. If we deal with  $n \times n$ matrices we have $n + 1$ classes of \zmas\ each consisting of matrices of the same dimension. In contrast the expression {\it class of $M$-matrices } includes always  matrices of different dimensions. 

It is well-know that $DA$ and $AD$ are again \mmas, if $A$ is an \mma\ and if
$D$ is a positive diagonal matrix. Hence, the classes $L_s$ are invariant under multiplying with a positive diagonal matrix \cite{FieM92}. 
Moreover, the classes $L_s$ are invariant under permutational similarity \cite{FieM92}. 

A matrix $A$ is called an inverse $Z$-matrix, if it is nonsingular and its inverse is a $Z$-matrix. More preciously, $A$ is called an inverse $L_s$-matrix, if it is nonsingular
and $A^{-1} \in L_s$. 

In \cite{Nab97} many properties of all  classes of $Z$-matrices and inverse $Z$-matrices are given. It turned our that it is important whether 
$\nhalf \leq s < n$ or not. One example is the determinant of a $Z$-matrix $A \in L_s$.

It is well-known, that $det A \geq 0$, if $A \in L_n$, i.e. $A$ is an
\mma. G.A. Johnson \cite{JohG82} proved that $det A \leq 0$ if $ A \in N_0$.
Moreover, the determinant of a \fma\ is nonpositive also \cite{JohG82}. However,
G.A. Johnson gave an example that the sign of the determinant of matrices  
in $L_{n-3}$ need not to be the same, \cite{JohG82}. But we have (see \cite{Nab97}):

\begin{theorem} \label{detnew}
Let $A \in L_s$.  Then\\
i) $det A \geq 0$, if $ s = n$, \\
ii) $det A \leq 0$, if $\nhalf \leq s < n$.
\end{theorem}

The counterexamples G.A. Johnson gave in \cite{JohG82} were $4 \times 4$ matrices
in $L_1$. So these counterexamples  do not refute
Theorem \ref{det}.

As mentioned in \cite{FriN97} the statement of Theorem \ref{detnew} can not be improved. Thus, determinants of different matrices
in $L_s$ with $s < \nhalf$ can  have different signs.


We  do have  the following well-known results.

\begin{theorem}[\cite{BerP94} Theorem 6.2.3]
 Let $A$  be a $Z$-matrix. Then  the following are  equivalent:
 \begin{enumerate}
  \item $A$ is a nonsingular $M$-matrix,
  \item All  the principal minors of $A$ are positive,
  \item $A^{-1}$ exists and $A^{-1} \geq 0$. 
 \end{enumerate}
\end{theorem}

Moreover

\begin{theorem}[\cite{Joh82} Theorem 2.7]
 Let $A$  be a $Z$-matrix. Then  the following are  equivalent:
 \begin{enumerate}
  \item $A$ is a nonsingular $N_0$-matrix,
  \item   $A^{-1}$ exists,     $A^{-1} \leq 0$ and is irreducible. 
 \end{enumerate}
\end{theorem}

The next Theorem gives a characterization of matrices in $L_s$ using their inverses (see \cite{Smi95} and \cite{Nab97}).

\begin{theorem} \label{equiformu}
Let $ A \in \bRn$ be a nonsingular $Z$--matrix. Then $A \in L_s$ if and only
if one of the following alternative cases a) or b) holds
\begin{enumerate}
\item[a)] \begin{enumerate} \item[i)]   $det A < 0$, 
\item[ii)] all principal minors of $A^{-1}$ of order greater or equal $n-s$ are nonpositive,
\item[iii)]   there exists a principal minor of $A^{-1}$ of order $n-s-1$, which is positive; \end{enumerate}
\item[b)] \begin{enumerate} \item[i)]   $det A > 0$, 
\item[ii)]   all principal minors of $A^{-1}$ of order greater or equal $n-s$ are nonnegative,
\item[iii)]   there exists a principal minor of $A^{-1}$ of order $n-s-1$, which is negative.
\end{enumerate}
\end{enumerate}
\end{theorem}

\subsection{Acyclic matrices}

Klein  established in \cite{KLe82} some results on matrices 
whose graphs are trees  or in other words, are irreducible and acyclic. We will use these results in the next section.

But  we first   need some graph theoretical notations and definitions. Here  we follow \cite{Ha72}. An {\it (edge) weighted  graph
G = (V,E)} of $n+1$ vertices is a graph with vertex set $V = \{0, \ldots, n\}$ and 
edges $e_{ij} \in E$ between the 
vertices ($i,j \in \{0,1,\ldots,n\}$) labeled by (nonzero) weights $w_{i,j} \in \bR$.


A {\it path} 
from vertex $i$ to vertex $j$, denoted by $P_{i,j}$, is a  ordered sequence of edges 
$\left( (k_1,k_2), (k_2, k_3), \ldots (k_{r-1},k_r)\right) $, where $k_1 = i$ and
$k_r = j$. The path $P_{i,j}$ is a cycle , if $k_1 = i = j = k_r$, $r \geq 3$ and $k_1, \ldots, k_{r-1}$ are distinct. 
For a vertex $i$ the neighborhood of $i$ is the set $N(i) := \{j \in V | (i,j) \in  E\}.$ The  cardinality of $N(i)$ is the degree of $i$ and is denoted by $deg(i)$.  

A graph is acyclic if it has no cycles. A {\it tree} is a connected acyclic 
graph. Equivalently, a tree is a graph for which there exists a
unique path between  any two vertices $i$ and $j$. 
A {\it rooted tree}  is a tree with a prominent vertex called the root.

Here we distinguish between trees $\Gamma $ and rooted trees $\Gamma_{(0)}$. 
We always assume that trees have vertex set $\{1,\ldots,n\}$ while for 
rooted trees $\{0,1,\ldots, n\}$ is the vertex set and the root is labeled by $0$. 
For a given tree $\Gamma$ we 
construct a (special) rooted tree $\Gamma_{(0)}$ by adding a new vertex $0$, the root,
and a new edge $e_{0,1}$ from the root to vertex 1 to the old tree. \\


We number the vertices of the trees in the following way. 
We start with an 
arbitrary vertex which is then labeled with $1$. 
We then proceed 
recursively by a depth-first 
search or numbering (dfs), see \cite{AHU}.

A {\it branch} starting at vertex $i$ of the  tree $\Gamma$  is the connected subgraph of $\Gamma$ including vertex $i$ obtained by deleting the unique edge $e_{j,i}$ with $ j < i $. Vertex $i$ then becomes the root of this particular branch. Thus  we label the vertices recursively branch by branch.

The {\it graph} $G(A)$ 
of an $n \times n$ matrix $A = [a_{i,j}]$ is the undirected graph
consisting of $n$ vertices $\{1,\ldots,n\}$ such that there is an edge between
vertex $i$ and vertex $j$ if and only if $a_{i,j} \neq 0 $ or $a_{j,i} \neq 0 $.

$A$ is called {\it treediagonal} by Klein \cite{KLe82} if $G(A)$ is 
a forest, i.e. a collection of trees. Here we prefer the name 
{\it acyclic matrices} for matrices $A$ whose graphs  
$G(A)$ are forests.  We always consider irreducible acyclic matrices $A$, thus $G(A)$ is a tree.

\begin{definition}
An $n \times n$ matrix $C = [c_{i,j}]$ satisfies the  {\it treeangle property} with respect to  a given tree $\Gamma$ if for every $i,j,k \in V$ with $P_{i,k} \subseteq P_{i,j}$ it holds
\beq \label{treeangle}
 c_{i,j}c_{k,k} = c_{i,k}c_{k,j}.
\eeq
Here $P_{i,j}$  denotes the unique path between the vertices $i$ and $j$ in the tree $\Gamma$. 
\end{definition}

Klein then proved the following theorems 

\begin{theorem} \label{kl1}
Let $\Gamma$ be a tree. If a non-singular matrix $C$ satisfies the treeangle
property with respect to $\Gamma$ and if $c_{i,i} \neq 0$ for all interior
vertices, then $C^{-1}$ is treediagonal with respect to $\Gamma$, i.e. 
$C^{-1}$ is acyclic. The entries of $C^{-1}$  are given by
\begin{equation} \label{formula-inv}
(C^{-1})_{ij} = \left\{ \begin{array}{ll}
-\frac{c_{ij}}{d_{ij}} & i \neq j, (i,j) \in E,\\
\frac{c_{kk}}{d_{ij}} & i = j, deg(i) = 1, (i,k) \in E,\\
\frac{1}{c_{ii}}( 1 + \sum_{k \in N(i)} \frac{c_{ik}c_{ki}}{d_{ik}} & i = j, deg(i) \geq 2,\\
0 &  otherwise
\end{array} \right.
\end{equation}
where
\begin{equation} \label{def-dij}
d_{ij} := c_{ii}c_{jj} -  c_{ij}c_{ji}.
\end{equation}
\end{theorem}

\begin{theorem} \label{kl2}
Let $\Gamma$ be a tree. If $A$ is a non-singular treediagonal matrix with respect to $\Gamma$ then $A^{-1}$ satisfies the  treeangle property with respect to 
$\Gamma$.
\end{theorem}

Moreover we have (see \cite{Nab01})

\begin{proposition} \label{Nab01}
Let $A =[a_{i,j}] \in \bRn $ be non-singular, irreducible and acyclic. Assume
that the diagonal entries of $C= [c_{i,j}] := A^{-1}$ are nonzero. Then 
$c_{i,j} \neq 0$ for all $i, j \in \{1,\ldots, n \}$. Hence, if $A$ is  irreducible, acyclic and diagonally dominant, then all entries of $C^{-1}$ are different from zero.  
\end{proposition}


The above Theorems  describe the structure of inverses of acyclic 
matrices. However, it is not clear at all how one can describe 
the treeangle property in terms of matrices. To do so we need the following class of matrices, see \cite{Nab01}.

\begin{definition}
Let $\Gamma_{(0)}$ be a weighted rooted tree with vertex set $\{0,1,\ldots, n\}$, where $0$ denotes the root. A matrix $A = [a_{i,j}] \in \bRn$ is of 
{\it tree structure} with respect to $\Gamma_{(0)}$, if for all $i,j = 1,\ldots,n$
\beq \label{dis}
a_{i,j} & = & \sum_{\{r,s\} \in P_{i,0} \cap P_{j,0}} w_{r,s},  
\eeq
here $\{r,s\} \in P_{i,0} \cap P_{j,0}$ denotes a common edge of the paths 
$P_{i,0}$ and  $P_{j,0}$  and $w_{r,s}$ is its weight.
\end{definition}

Note that (\ref{dis}) defines a 'distance' or better an 'inverse distance' 
between the vertices of the tree. This distance was already used  
in \cite{NabV95b}. \\
If all   weights  are non-negative, matrices of tree structure are ultrametric matrices \cite{NabV94, VarN93, Fie00}.  We can also define non-symmetric matrices of tree structure by introducing a second set of weights which are used to build the $a_{j,i}$ similar as in (\ref{dis}). This than lead to generalized ultrametric matrices \cite{NabV95a, NabV95b, Nab97, McDNNST98, ElsNN98, DelMM14}. 
The next  subsection describes ultrametric matrices in detail. The last section gives some examples  for  matrices of tree structure.

In the following we give some properties of matrices of tree structure. Some of them are new and some of them are already given in \cite{Nab01}.

Theorem 3.3 in \cite{Nab01} is the following one. 

\begin{theorem} \label{iff-2}
Let $\Gamma$ be a tree and let $A = [a_{i,j}] \in \bRn$ be nonsingular. 
Then the following are equivalent:

\begin{itemize}
\item[(1)] $G(A) = \Gamma$ and $A\xi_n = \gamma \hat e_n$ and $\xi_n ^TA = \gamma \hat e_n^T$.

\item[(2)] $A^{-1}$ is symmetric and $A^{-1}$ is of tree structure 
with respect to the weighted rooted tree $\Gamma_{(0)}$ where the weights are 
$w_{i,j} = -1/a_{i,j}$ and $w_{0,1} = 1/\gamma$. 
\end{itemize}  
\end{theorem}

\begin{remark}
Note, since $\xi_n$ is the vector of all ones  and $\hat e_n$ is the vector with a one at first position and zeros else where, the first  statement of the above Theorem says, that the entries of the first  row and column of $A^{-1}$ are all  the same. 
Moreover, note  that the  reverse ordering of the vertices of a graph gives also a matrix of tree structure. Then  we have $A\xi_n = \gamma \tilde e_n$ and $\xi_n^TA = \gamma \tilde e_n^T$ and a similar statement as in Theorem \ref{iff-2} holds. 
\end{remark}

Theorem 3.9  in \cite{Nab01} is then
\begin{theorem} \label{theo-tree-row-2}
Let $C = [c_{i,j}] \in \bRn$ be of tree structure with respect to a given rooted tree. 
Then $C$ is nonsingular if and only if $C$ does not contain a row or column of zeros, and no two rows or two columns are the same. 
\end{theorem}

Both of the above theorems can be extended.
\begin{theorem} \label{theo:nonzero-w}
Let $C = [c_{i,j}] \in \bRn$ be of tree structure with respect to a given rooted tree. 
Then $C$ is nonsingular if and only if there is no zero  weight in the weighted tree.
\end{theorem}
\begin{proof}
  We use Theorem \ref{theo-tree-row-2}.
  If there is a zero weight in the  tree then obviously two rows and columns are the same  and $C$  is singular.
  Now assume that  there is no zero weight. We then show that there  are no  two columns which are the same. If there would be such columns, then we need  to have two diagonal entries which are the same. So we obtain 
  \[
  a_{ii} = a_{jj} \quad \mbox{for some indices $i,k,j$ with } \quad i < k < j,
  \]
  since  we do not have zero weights. Moreover, $w_{kj} \neq 0$.
  But due to the tree structure we then have
  \[
  a_{ii} = a_{ij} =  a_{jj}.  
  \]
  But since  $w_{kj} \neq 0$  we obtain 
  \[
  a_{ik} \neq a_{jk}.
  \]
  Thus  there are no two columns which are the same.
  \end{proof}

Theorem \ref{theo:nonzero-w} guarantees that the weights  are not zero for nonsingular matrices. So Theorem \ref{iff-2} then gives simple formulas for the inverses of the considered matrices in both directions. 
The main theorem in \cite{Nab01} is the  following Theorem (see Theorem 3.4 in \cite{Nab01}). 

\begin{theorem} \label{main-structure}
Let $A \in \bRn$ be non-singular, irreducible and acyclic and let $\Gamma = G(A)$.  Assume that the diagonal entries of $A^{-1}$ are nonzero. 
Then there exist matrices  $T$ and $R$ of the form 

\beq \label{def-neu}
 T = \left[ \begin{array}{ccccc}
d_1 & d_1 & \cdots &\cdots & d_1  \\
f_1 & d_2 & \cdots &\cdots & d_2 \\
f_1 & f_2 & d_3 & \cdots & d_3 \\
 \vdots & \vdots & & \ddots & \vdots \\
f_1 & f_2 & \cdots & \cdots & d_n
\end{array} \right] \quad \quad
R = \left[ \begin{array}{ccccc}
f_1 & f_2 & \cdots & \cdots & f_n  \\
d_2 & f_2 & \cdots & \cdots & f_n \\
d_3 & d_3 & f_3 & \cdots & f_n \\
 \vdots & \vdots & & \ddots & \vdots \\
d_n & d_n & \cdots & \cdots & f_n
\end{array} \right],
\eeq
with $f_1 = d_1$, and a matrix $U$ of tree structure with respect to the weighted tree $\Gamma_{(0)}$ such that

\beqo
A^{-1} = T \circ U \circ R.
\eeqo
The sequences $\{f_i\}$ and $\{d_i\}$ are given by 
$F = diag(f_1,\ldots,f_n) = diag(\hat e_n^TA^{-1})$ and 
$D = diag(d_1,\ldots,d_n) = diag(A^{-1}\hat e_n)$.

Conversely, let $U = [u_{i,j}] \in \bRn$ be non-singular. If $U$ is of tree structure with respect to a given rooted
tree $\Gamma_{(0)}$. Then for any matrices $T$ and $R$ of the form 
$(\ref{def-neu})$ with $f_1 = d_1$ and  $f_id_i \neq 0$ for all $i$, the matrix 
$(T \circ U \circ R)^{-1}$ is irreducible, acyclic and $G((T \circ U \circ R)^{-1}) = \Gamma.$ Moreover
\begin{equation} \label{rowsum}
 U\xi_n = \gamma \hat e_n, \ \xi_n^TU = \gamma \hat e_n \ \mbox{for } \gamma \in \bR, \  \xi_n^T = [1, \ldots, 1] \in \bR^n, \hat e_n = [1,0 \ldots , 0]^T.
\end{equation}
\end{theorem}

Note again that here $\Gamma_{(0)} $ is obtained from $\Gamma$ by adding the root $0$ and a 
new edge $e_{0,1}$. Moreover, note that 
it also holds
\beq \label{eq:prod}
A^{-1} = D*U*F.
\eeq

For tridiagonal matrices the above result leads to  the classical results given in  \cite{GK1, GK2, I}, see also \cite{Nab01}. 

In the symmetric case the sequences $\{f_i\}$ and $\{d_i\}$ are the same. Then  the matrix $T$ is called a type D matrix and the matrix $R$ is called a flipped type D matrix, see \cite{Mar72, McDNNST98}. Note that originally (\cite{Mar72, McDNNST98})  in the definition of these matrices the sequence of the diagonal entries must be decreasing or increasing. Here, as in \cite{Nab97},  we omit this condition.

In \cite{PraS26} inverse $M$-, inverse $N_0$- and inverse $F_0$-  matrices are characterized whose graph is  a tree (or are acyclic). 
However, these characterization  are not that simple  to check and do not lead to a simple way to construct these matrices. In Section 3  we  give simple  characterizations of all inverse $L_s$-matrices whose graphs are trees. To  do so we use matrices of tree structure.

\subsection{Generalized ultrametric matrices and shifted  generalized ultrametric matrices}

In \cite{MarMS94}  the class of strictly ultrametric matrices is  introduced. It  is then shown in \cite{MarMS94} and \cite{NabV94} that their inverses are diagonally dominant $M$-matrices. A more general definition leads to ultrametric matrices.

\begin{definition}\label{defgultra}
A symmetric matrix $A  \in \bRn$ is an ultrametric matrix  if
\begin{enumerate}
\item[i)]  $A$ has nonnegative entries;
\item[ii)]  $a_{i,j} \geq min\{a_{i,k};a_{k,j}\}$ for all $i,k,j \in N;$ 
\item[iii)] $a_{i,i} \geq max \{a_{i,j}; 
a_{j,i}\}$ for all $i,j \in N$.
\end{enumerate}
A matrix $A$ is said to be a {\it strictly ultrametric  } if strict inequality holds in
 iii) for all $i \neq j$, where, if $n = 1$, this is interpreted as $a_{1,1} > 
0$.\end{definition} 

In \cite{McDNST95}  and
\cite{NabV95a} nonsymmetric generalizations of ultrametric matrices are introduced.

\begin{definition}\label{defgenultra}
A matrix $A  \in \bRn$ is a \gum\  if
\begin{enumerate}
\item[i)]  $A$ has nonnegative entries;
\item[ii)]  $a_{i,j} \geq min\{a_{i,k};a_{k,j}\}$ for all $i,k,j \in N;$ 
\item[iii)] each triple $\{q,s,t\} \subseteq N^3$ can be
      reordered as a triple $\{i,j,k\}$  such that
   \begin{enumerate} 
     \item[$1)$]  $a_{j,k} = a_{i,k}$  and  $ a_{k,j} = a_{k,i}$;
     \item[$2)$] $max\{a_{i,j};a_{j,i}\} \geq max\{a_{i,k};a_{k,i}\}$; 
    \end{enumerate}
\item[iv)] $a_{i,i} \geq max \{a_{i,j}; 
a_{j,i}\}$ for all $i,j \in N$.
\end{enumerate}
A matrix $A$ is said to be a {\it strictly \gum\ } if strict inequality holds in
 iv) for all $i \neq j$ in $N$, where, if $n = 1$, this is interpreted as $a_{1,1} > 
0$.\end{definition} 
It is then shown in 
\cite{NabV95a} and \cite{McDNST95}  that the inverses of nonsingular generalized ultrametric matrices are diagonally dominant $M$-matrices. 

As shown in \cite{NabV95a} \gums \  and ultrametric matrices  can be represented with a (edge) weighted rooted tree with $m$ leaves different from the root. The  edges of the tree have two weights (say left ones) $l_{i,j}$ and (right ones) $w_{i,j}$.

\begin{theorem}
A matrix $A = [a_{i,j}] \in \bR ^{m,m}$ is a generalized ultrametric matrix if and only if there  exits a weighted rooted tree $\Gamma_{(0)}$ with $m$ leaves (different from the root) and vertex set $\{0,1,\ldots, m\}$, where $0$ denotes the root.
Each edges of the tree has two nonnegative weights (say left ones) $l_{i,j}$ and (right ones) $w_{i,j}$  such that for all $i,j = 1,\ldots,m$
\beq \label{disgen}
a_{i,j} & = & \sum_{\{r,s\} \in P_{i,0} \cap P_{j,0}} w_{r,s} \quad \mbox{for} \quad i < j \\
a_{i,j} & = & \sum_{\{r,s\} \in P_{i,0} \cap P_{j,0}} l_{r,s} \quad \mbox{for} \quad i > j \\
a_{i,i} & \geq  & \max \{ \sum_{\{r,s\} \in P_{i,0}} l_{r,s}, \sum_{\{r,s\} \in P_{i,0}} w_{r,s} \},
\eeq
here $\{r,s\} \in P_{i,0} \cap P_{j,0}$ denotes a common edge of the paths 
$P_{i,0}$ and  $P_{j,0}$.
\end{theorem}

Note in the theorem above we consider the leafs of the tree.  But this construction can be easily generalized for all vertices (leafs and interior vertices). Formally, for every interior vertex we add a new leaf and a new edge with zero weights which connects both. Thus we can easily obtain a rooted tree with $n$ leafs from $\Gamma_{(0)}$, (where $n$ is the number of all vertices in $\Gamma$). 

Hence, we can see a close relation between symmetric ultrametric matrices and nonnegative matrices of tree structure  as mentioned above. Moreover, symmetric shifted ultrametric matrices are matrices of tree structure. But matrices of tree  structure need not to be shifted ultrametric. This depends on the weights of the tree.

Generalized ultrametric matrices have a special nested block structure. To describe this structure we  need  the following notation.  For  $A = [\aij] \in \bRn, $ we set

\beq \label{defentries}
\left\{ \begin{array}{cll}
\tau (A) & := & min \{\aij : i,j \in N\}, \\ 
\omega (A) & := &  min \{\aji : \aij = \tau (A)\}, \\
\tom & := & \omega (A) - \tau (A), \ \mbox{ so that } \ \tom \geq 0. 
\end{array} \right.
\eeq

We see that if $A$ is a \gum\  in $\bRn$, then the matrix $B := A + 
\tau\xi_n\xi_n^T$ 
satisfies ii), iii), and iv) of Definition \ref{defgenultra} for {\it any} real 
$\tau$. This observation gives rise to the next class of matrices. (Definition 2.2 in \cite{NabV95b}).

\begin{definition}\label{defessgen}
A matrix $A = [a_{i,j}] \in \bRn$ is a \sgum\  if there is a real number $\tau$ 
such that $A + \tau\xi_n\xi_n^T$  is a \gum. If strict 
inequality holds in iv)  of Definition $\ref{defgenultra}$ for all $i \in N$, 
then $A$ is a strictly \sgum. In addition, if $A = [a_{i,j}] \in \bRn$ is a 
\sgum\  which is not a \gum, so that $\tau(A)<0$, then $A$ is said to be of 
type $U^{(-1)}_{p,n-p}$ if
\begin{enumerate}
\item[v)]  
           $a_{i,j} \leq 0$ for all $i,j \in N$, with $i \neq j$; 
\item[vi)] $p$ is the largest positive integer such that there exists a $p 
\times p$ principal submatrix of $A$ which is a nonsingular M--matrix. If no 
such positive integer exists, then $p := 0$.
\end{enumerate}
\end{definition}

The next remark compares shifted  generalized ultrametric matrices and matrices of tree structure.

\begin{remark}
 The classes of shifted generalized ultrametric matrices and  matrices of tree structure share many nice properties. But the classes are different - but of course with a huge overlap. We can always add the identity to a shifted generalized ultrametric matrix and the result is still of this type. But for matrices of tree structure this does not work. 
 So for matrices of tree  structure we always have off-diagonal entries which are the same as diagonal entries. On the other hand 
 if we have positive and negative weights for interior vertices in the tree, the matrix of tree structure is not a shifted generalized ultrametric matrix. For ultrametric matrices the entries 'grow in the direction to the diagonal'. 
\end{remark}

Throughout this paper, we use the following theorems which are all proven in  \cite{NabV95a} or \cite{NabV95b}. We start with the  Theorems 
3.7 and 3.8 of \cite{NabV95b}.

\begin{theorem} \label{theopreultra}
Let $A = [\aij] \in \bRn $ be a strictly \gum. Then, $A$ is nonsingular and its 
inverse $A^{-1} \in \bRn$ is a strictly row and strictly column diagonally 
dominant M--matrix, with the property that
\beq \label{omegaabst}
\omega (A) \xi^T_nA^{-1}\xi_n < 1.
\eeq
If $A$ is a nonsingular \gum,  its inverse $A^{-1} \in \bRn$ is a row and column
 diagonally dominant M--matrix, with the  
property that
\beq \label{omegaabw}
\omega (A) \xi^T_nA^{-1}\xi_n \leq 1.
\eeq
\end{theorem}
The next Theorem is  Theorem 2.3 in \cite{NabV95b}.

\begin{theorem} \label{nested}
Let $A = [\aij] \in \bRn, n > 1, $ be a \sgum. 
 Then, there exist a permutation matrix $P \in \bRn$ and a positive integer
$r$, with $1 \leq r < n$, such that
\beq  \label{eqrep}
 \mbox{~~~~~~~} PAP^T & = & 
\left[ \begin{array}{cc}
A_{1,1} & A_{1,2} \\
A_{2,1} & A_{2,2}
\end {array} \right] =
\left[ \begin{array}{cc}
C & O \\
\tom \xi _{n-r}\xi ^T_r & D 
\end {array} \right] + \tau (A)\xi_n \xi _n^T \\  \nonumber
& =:  & M + \tau (A)\xi_n \xi _n^T,
\eeq
where $A_{1,1}$ $ \in \bR ^{r,r} $  and $A_{2,2} \in \bR^{n-r,n-r}$ 
are \sgums. Moreover, the matrices $A_{1,1}$ and $A_{2,2}$ are reduced in the 
same way as $A$. The off--diagonal blocks are $A_{1,2} = \tau 
(A)\xi_r\xi^T_{n-r}$ and $A_{2,1} = \omega (A)\xi_{n-r}\xi^T_{r}$. \\
 

The matrices $M \in \bRn,\ C \in \bR ^{r,r} $ and $D \in \bR^{n-r,n-r}$
  are generalized ultrametric matrices and $\tom = \omega (A) - \tau (A) 
\geq 0.$  
Moreover,
\beq \label{eqomega}
\omega (A_{1,1}) \geq \omega (A),\quad  {\rm and} \quad \omega (A_{2,2}) \geq 
\omega (A).
\eeq
If $A$ is a strictly \sgum, then $M$ is a strictly generalized ultrametric 
matrix.
\end{theorem}

The next Lemma is Lemma 2.4 in \cite{NabV95b}. 

\begin{lemma}\label{strict}
Let $A \in \bRn, n > 1, $ be a \sgum\ given in the block form $(\ref{eqrep})$.\\
i) If $A$ is a nonsingular strictly \sgum\ of type $U^{(-1)}_{p,n-p}$ for some 
$p$ with $0 \leq p < n,$ then
\beq \label{abschatz}
A^{-1}\xi_n < 0 \quad \mbox{ and } \quad \omega (A) \xi_n^TA^{-1}\xi_n > 1.
\eeq 
Similarly, if $A$ is a nonsingular \sgum\ of type $U^{(-1)}_{p,n-p}$ for some 
$p$ with $0 \leq p < n,$ then
\beq \label{abschatzweak}
A^{-1}\xi_n \leq 0 \quad \mbox{ and } \quad \omega (A) \xi_n^TA^{-1}\xi_n \geq 
1.
\eeq  
ii)
If $A_{1,1}$ in $(\ref{eqrep})$ is nonsingular and of type $U^{(-1)}_{p,r-p}$ 
for some $p$ with $0 \leq p < r,$
 then the Schur complement $A/A_{1,1}$ is a generalized ultrametric matrix. If 
$A_{1,1}$ is a nonsingular M--matrix, i.e., of type $U^{(-1)}_{r,0}$,
and if $A_{2,2} \in U^{(-1)}_{q, n-r-q}$ with $\ 0 \leq q \leq n-r$, then 
$A/A_{1,1}$ is of type $U^{(-1)}_{\tilde p, n-r-\tilde p}$ for a $\tilde p$ with
 $0 \leq \tilde p \leq q$.
The same holds for $A/A_{2,2}$.
\end{lemma}

The next Theorem is Theorem 2.5 in \cite{NabV95b}.

\begin{theorem} \label{nonsing}
Let $A = [\aij] \in \bRn $ be a \gum \ or a \sgum\ with a 
negative diagonal entry. Then, $A$ is nonsingular if and only
if $A$ contains no 
zero row or zero column and no two rows or two columns which are the same. 
\end{theorem}

Theorem 2.7 in \cite{NabV95b} is the following

\begin{theorem} \label{det}
Let $A = [\aij] \in \bRn $ be a nonsingular \sgum\ of type $U^{(-1)}_{p,n-p}$. If
 $0 \leq p < n$, then
\beqo
det A < 0.
\eeqo
Moreover, each principal minor of order $m$ with $p < m < n$ is nonpositive,
and there exists a positive principal minor of order $p.$ If $p=n$ (i.e., $A$ is
of type $U_{n,0}^{(-1)})$ or if $A$ is a nonsingular 
 \gum, \ then $det A > 0$.
\end{theorem}

Finally we have Theorem 2.9 in \cite{NabV95b}.

\begin{theorem} \label{inverse}
Let $A = [\aij] \in \bRn $ be a nonsingular \sgum. \ If A is a \gum,\ 
then $A^{-1}$ is in $L_{n}$. If $A$ is of 
type $U^{(-1)}_{n-m,m}$, for an $m$ satisfying $1 \leq m \leq n$, then 
$A^{-1}$ is in $L_{m-1}$, i.e., $A^{-1}$ is of the form {\rm (cf. (1.1))}
\beqo
A^{-1} = tI - B,
\eeqo
with $B \geq O$ and $\rho_{m-1}(B) \leq t < \rho_m(B).$ Moreover, $\rho_{m-1}(B)
 = t$ if and only if there exists a principal submatrix, of order 
$n-m+1$, in $A$ which is singular.
\end{theorem}

Thus, nonsingular shifted generalized ultrametric matrices of type $U^{(-1)}_{0,n}$, $(p = 0)$, are inverse $N_0$ matrices.

\section{Inverse acyclic $Z$-matrices}

We  start with  the following characterization of symmetric inverse acyclic $M$-matrices.

\begin{theorem}\label{invM}
  Let $ A = [a_{i,j}] \in \bRn$ be symmetric, irreducible and of tree structure  for a given tree $\Gamma_0$ with weights $w_{i,j}$. Then the following are equivalent
  \begin{itemize}
  \item $A$ is nonsingular and $A^{-1}$ is an M-matrix,
  \item all  weights $w_{i,j}$  are positive.
  \end{itemize}
\end{theorem}
\begin{proof}
  If  $A$ is nonsingular then with Theorem \ref{kl1} the off-diagonal entries different from zero are given  by
  \beq \label{eq:off}
  (A^{-1})_{i,j} = -\frac{a_{i,j}}{d_{i,j}}.
  \eeq
  If  $A^{-1}$ is an $M$-matrix, the off-diagonal entries are nonpositive, but the entries of $A$ must be nonnegative. Thus, the $d_{i,j}$ in (\ref{def-dij})  must be positive. So
  \beqo
  d_{i,j} = a_{i,i}a_{j,j} -  a_{i,j} a_{j,i} > 0.
  \eeqo
  Now, assume w.l.o.g. that $i < j$. Then  due  to the tree structure we have
  \beqo
  a_{i,i} = a_{i,j} = a_{i,j}.
  \eeqo
  If there would be a  nonpositive weight  $w_{i,j}$, then
   \beqo
 0 \leq a_{j,j} \leq a_{i,j},
  \eeqo
  which implies 
\beqo
  d_{i,j} \leq 0.
  \eeqo
  But this is a contradiction. So all  weights are positive.
  
  Conversely, if all   weights are positive, then $A$ is a \stum  \ and no two rows or columns are the same, hence $A$ is nonsingular and the inverse is an $M$-matrix.
  \end{proof}

Next we give a characterization of symmetric inverse acyclic $N_0$-matrices.

\begin{theorem}\label{invN0}
  Let $ A = [a_{i,j}] \in \bRn$ be symmetric, irreducible and of tree structure  for a given tree $\Gamma_0$ with weights $w_{i,j}$. Then the following are equivalent
  \begin{itemize}
  \item[1.] $A$ is nonsingular and $A^{-1}$ is an $N_0$-matrix,
  \item[2.] \begin{enumerate}
    \item[(a)]  $w_{0,1} < 0$
    \item[(b)]  $w_{i,j} > 0$ for all $i,j > 0$
      \item[(c)] For all $ j \in L$ it holds $|w_{0,1}| > \sum_{i \in P_{1,j}}w_{i,j} $, where $L$ denotes the set of leaves of $\Gamma_0$.
  \end{enumerate}
    \end{itemize}
\end{theorem}
\begin{proof}
Note, 
  If  $A$ is nonsingular then with Theorem \ref{kl1} the off-diagonal entries different from zero are given  by
  \beqo
  (A^{-1})_{i,j} = -\frac{a_{i,j}}{d_{i,j}}.
  \eeqo
  If  $A^{-1}$ is an $N_0$-matrix, the off-diagonal entries are nonpositive, and the entries of $A$ are also nonpositive.  Thus we need to have  $w_{0,1} < 0$ and $2 (c)$. The $d_{i,j}$  must be negative. So
  \beqo
  d_{i,j} = a_{i,i}a_{j,j} -  a_{i,j} a_{j,i} < 0.
  \eeqo
  Now, assume w.l.o.g. that $i < j$. Then  due  to the tree structure we have
  \beqo
  a_{i,i} = a_{i,j} = a_{i,j}.
  \eeqo
  If there would be  a nonpositive weight  $w_{i,j}$, then
   \beqo
|a_{j,j}| > |a_{i,j}|,
  \eeqo
  which would  give 
\beqo
  d_{i,j} > 0,
  \eeqo
  which is a contradiction.
  Conversely, conditions (a) and (c) imply that $A < 0$.  Then conditions (a) and (b)  imply that the $d_{i,j}$  are negative. Moreover, $A$ is a shifted generalized ultrametric matrix and due to the tree structure no two rows or columns are the same. Hence  $A$ is nonsingular.  Thus $A^{-1}$ is a $Z$-matrix which is an $N_0$-matrix.
    \end{proof}

    It  is a bit more complicated to determine the weights of the tree  to obtain an inverse $Z$-matrix which belongs to other classes $L_s$. But  the next Theorem gives  a way to construct such a matrix. 
    
\begin{theorem} \label{theo:N-M}
  Let $ A \in \bRn$ be of tree structure of the form    
  \beq \label{eqrepn}
PAP^T =
\left[ \begin{array}{cc}
A_{1,1} & A_{1,2} \\
A_{2,1} & A_{2,2}
\end {array} \right] 
\eeq
where $A_{1,1} \in \bR ^{s,s} $ is an inverse $N_0$-matrix and $A_{2,2} \in \bR^{n-s,n-s}$ 
is an inverse $M$-matrix. The off diagonal blocks are $A_{1,2} = \tau\xi_s\xi^T_{n-s}$ and $A_{2,1} = \tau \xi_{n-s}\xi^T_{s}$, where 
$\tau = \tau(A_{1,1})$.
Then $A$ is nonsingular and $A^{-1}$ is an acyclic $Z$-matrix with 
$A^{-1} \in L_{s-1}$.
    \end{theorem}
\begin{proof}
 It is obvious that $A$ is of tree structure and a shifted \gum \ with $det A < 0$. Moreover, $A$ is nonsingular since no two rows and columns are the same.
 By construction we have for all $i = 1, \ldots, s$ and $j = 1, \ldots, n$  and $i = s + 1, \ldots, n$ and $j = 1, \ldots, s$
 \beqo
 0 > d_{i,j} = a_{i,i}  a_{j,j} - a_{i,j} a_{j,i}.
 \eeqo
 But $a_{i,j} < 0$  for these $i,j$. For 
 for all $i = s + 1, \ldots, n$ and $j = s + 1, \ldots, n$ we have 
 \beqo
 d_{i,j} = a_{i,i}  a_{j,j} - a_{i,j} a_{j,i} > 0.
 \eeqo
 But $a_{i,j} >0$  for these indices.
 Thus $A^{-1}$ is an $Z$-matrix.
 
 If we consider principal submatrices $\tilde A$ of $A$ of order greater  then $n-s$ then $\tilde A$ has a similar  block structure as above with $\tilde A_{1,1}$ being an inverse N-matrix and $\tilde A_{2,2}$ being an shifted generalized ultrametric matrix with negative determinant (it has a negative diagonal entry). 
 Then the Schur-complement $A \backslash A_{1,1}$ is a generalized ultrametric matrix  which has a positive determinant, but   $\tilde A_{1,1}$ has a negative determinant. Thus  $det \tilde A < 0$. Moreover, $det A_{2,2} > 0$. 
 Thus   with Theorem \ref{equiformu} we obtain $A^{-1} \in L_{s-1}$.
\end{proof}

The proof of Theorem \ref{theo:N-M} can be generalized to get a more general result. The $A_{1,1}$ block in (\ref{eqrepn}) can be of arbitrary dimension and $A_{1,1}$ can be an arbitrary inverse $L_t$-matrix.  Similarly $A_{2,2}$ can be chosen arbitrarily.  Then choose the $\tau$ such that $A$ is of tree  structure. The matrix $A$ is then an inverse acyclic  
$Z$-matrix of a certain class. The class depends on the dimensions of $A_{1,1}$ and $A_{2,2}$ and the classes their inverses belong to. 
 
 In the Theorems above, we have considered symmetric shifted \gums.  But the next result shows that shifted generalized ultrametric matrices whose inverse is acyclic are almost symmetric. To describe the structure we define so-called  shifted generalized cyclopes.
 
 \begin{definition}
   Let $ A = [a_{ij}] \in \bRn$. $A$ is a shifted generalized cyclopes (with eye $m$), if there  exists a permutation matrix $P$, such that
    \beq \label{eq:cyc}
P^TAP =
\left[ \begin{array}{cc}
A_{1,1} & \tau E_{1,2} \\
\omega E_{2,1} & A_{2,2}
\end {array} \right], 
\eeq
where $A_{1,1} \in \bRmm$ and $A_{2,2} \in \bRnmnm$ are square nonsingular matrices of tree structure, 
 $E_{1,2}$ and $E_{1,2}$ are all ones matrices and nonzero $\tau, \omega \in \bR$.


$A$ is in normal block form, if 
\[ P = I, \quad  A_{1,1}\tilde e_m = a_{m,m}\xi_m \quad \mbox{and} \quad A_{2,2}\hat e_{m-n} = a_{m+1,m+1}\xi_{n-m}. \]
Here $\tilde e_m := [0, \ldots, 0, 1]^T   \in \bR^m$   and $   \hat  e_{m-n} = [1, 0, \ldots, 0]^T \in \bR^{m-n}.$  
 \end{definition}

The above definition  generalizes the definition of cyclopes given in \cite{McDNNST98}  for tridiagonal  matrices. Moreover, it generalizes the definition in \cite{MarMZ03} where only positive (nonnegative) matrices are considered. 

We immediately  obtain
\begin{lemma} \label{lem:cyc-schur}
Let $A$ be a nonsingular  shifted generalized cyclopes  in normal block form. Assume  that $A_{1,1}$ and $A_{2,2}$ are nonsingular, then the Schur-complements $A \slash A_{1,1}$ and  $A \slash A_{2,2}$ are of tree structure. 
 \end{lemma}
\begin{proof}
 We have
 \beq \label{eq:cyc-schur}
A \slash A_{2,2} & = & A_{1,1} - A_{1,2}A^{-1}_{22}A_{2,1} = A_{1,1} - \tau \omega  \xi^T_{m-n}A^{-1}_{22}\xi_{m-n} \xi_m\xi_m^T. 
  \eeq
Thus  $A \slash A_{2,2}$ is of tree  structure. Similarly, one can prove that  $A \slash A_{1,1}$ is of tree  structure.
\end{proof}

The next Theorem  gives a characterization of the nonsingularity of a  shifted generalized cyclopes. 

\begin{theorem} Let $A$ be a nonsingular  shifted generalized cyclopes  in normal block form. Then $A$ is nonsingular if there is no zero row or column and no two rows or columns which are the same. 
 \end{theorem}
\begin{proof}
 Obviously  we just have to prove one implication. So assume that there are no zero rows and columns and no two rows and columns which are the same. But then with Theorem \ref{theo-tree-row-2} $A_{1,1}$ and $A_{2,2}$ are nonsingular. 
 Using  Lemma  \ref{lem:cyc-schur} the Schur-complement  $A \slash A_{1,1}$  is of tree structure. But  (\ref{eq:cyc-schur}) then gives that $A \slash A_{1,1}$  is also nonsingular. 
\end{proof}

Next we state a characterization of  shifted generalized cyclopes with eye $m$.

\begin{theorem} \label{lem:cyc}
 Let $A \in \bRn$ be a nonsingular. Then  the following are equivalent:
 \begin{enumerate}
  \item $A$ is a  shifted generalized cyclopes  in normal block form  with eye $m$.
 \item $A^{-1} = C$ is irreducible acyclic  and there exists the  edge $(m,m+1)$ in the tree of $A^{-1}$ such that $(\xi_n^TA^{-1})_j$ and $(A^{-1}\xi_n)_j$ are zero for all 
  $ j \in \{1, \ldots, n\} \backslash \{m,m+1\}$. Moreover, $A^{-1}(1:m)$ and $A^{-1}(m+1:n)$ are nonsingular.
   \end{enumerate}
  \end{theorem}
\begin{proof}
Note that we use the notation
\[ \tilde e_t := [0, \ldots, 0, 1]^T   \in \bR^t  \quad \mbox{and} \quad   \hat  e_t = [1, 0, \ldots, 0]^T \in \bR^t   . \] 

First prove {\bf 1. $\rightarrow $ 2. }  We have
   \beqo
A =
\left[ \begin{array}{cc}
A_{1,1} & \tau E_{1,2} \\
\omega E_{2,1} & A_{2,2}
\end {array} \right], 
\eeqo
where $A_{1,1} \in \bRmm$ and $A_{2,2} \in \bRnmnm$ are square nonsingular matrices of tree structure, 
 $E_{1,2}$ and $E_{1,2}$ are all ones matrices and nonzero $\tau, \omega \in \bR$. 
 As in Lemma \ref{lem:cyc-schur}  the Schur-complements are 
 \beqo
 A \slash A_{2,2} & = & A_{1,1} - A_{1,2}A^{-1}_{22}A_{2,1} = A_{1,1} - \tau \omega  \xi^T_{m-n}A^{-1}_{22}\xi_{m-n} \xi_m\xi_m^T,\\
  A \slash A_{1,1} & = &  A_{2,2} - A_{2,1}A^{-1}_{11}A_{1,2} = A_{2,2} - \tau \omega  \xi^T_{m}A^{-1}_{11}\xi_{m} \xi_{n-m}\xi_{n-m}^T.
   \eeqo
 Moreover, they are nonsingular and of tree structure.
 
 Then using the well-know formula for the inverse of a block $2 \times 2$ matrix  we get 
\beqo
 A^{-1} = 
\left[ \begin{array}{cc}
  (A \slash A_{2,2})^{-1}   &   A^{-1}_{11} A_{1,2} (A \slash A_{1,1})^{-1}  \\
 (A \slash A_{1,1})^{-1}A_{2,1}A^{-1}_{11}    &  (A \slash A_{1,1})^{-1}
\end {array} \right].
\eeqo
Since $(A \slash A_{2,2})$ is of tree structure and so $(A \slash A_{2,2})\tilde e_m = \gamma_2 \xi_m$ holds for some $\gamma_2 \in \bR$, we have
\[
 (A \slash A_{2,2})^{-1} \xi_m = \frac{1}{\gamma_2} \tilde e_m.
\]
Similarly
\[
 (A \slash A_{1,1})^{-1} \xi_{n-m} = \frac{1}{\gamma_1} \hat e_{n-m}.
\]
Moreover, $A_{1,1}^{-1}\xi_m = \alpha_2 \tilde e_m$  for some $\alpha_2 \in \bR$.
Hence,  
\[
  A^{-1}_{11} A_{1,2} (A \slash A_{1,1})^{-1} =  A^{-1}_{11}\tau \xi_m \xi_{m-n}^T (A \slash A_{1,1})^{-1} = \frac{\tau \alpha_2}{\gamma_1} \tilde e_m \hat e_{n-m}^T = \beta_1 \tilde e_m \hat e_{n-m}^T, 
\]
for some $\beta_1 \in \bRn$.
Similarly, 
\[
 (A \slash A_{1,1})^{-1}A_{2,1}A^{-1}_{11} = \beta_2 \hat e_{n-m} \tilde e_m^T,
\]
for some $\beta_2 \in \bR$.
Thus, $A^{-1}$ is acyclic and we have 
$(\xi_n^TA^{-1})_j$ and $(A^{-1}\xi_n)_j$ are zero for all 
  $ j \in \{1, \ldots, n\} \backslash \{m,m+1\}$.

 For   {\bf 2. $\rightarrow$  1.}
consider the block partition  of  $A^{-1}$: 
\beqo
 A^{-1} = C  =
\left[ \begin{array}{cc}
 C_{1,1} &   C_{1,2}  \\
 C_{2,1}  &  C_{2,2} 
\end {array} \right],
\eeqo
where $C_{1,1} \in \bRmm$ and $C_{2,2} \in \bRnmnm$. 
Since $C$ is irreducible we have 
\beqo
 C_{1,2} =
\left[ \begin{array}{cc}
 0 &   0 \\
 \alpha &  0 
\end {array} \right] = \alpha \tilde e_m \hat e_{n-m}^T, \quad \mbox{and} \quad C_{2,1} =
\left[ \begin{array}{cc}
 0 &   \beta \\
 0 &  0 
\end {array} \right] = \beta \hat e_{n-m} \tilde e_m^ T
\eeqo
for some nonzero $\alpha, \beta  \in \bR $. 
We have 
\beqo
C\xi_n = \left[ \begin{array}{c}
C_{1,1}\xi_m + \alpha \tilde e_m   \\
C_{2,2}\xi_{n-m} + \beta \hat e_{n-m}
\end {array} \right] =: 
\left[ \begin{array}{c}
\sigma_m \tilde e_m   \\
\sigma_{n-m} \hat e_{n-m}
\end {array} \right], 
\eeqo
\beqo
\quad \xi_n^TC = [ \xi_m^TC_{1,1}+ \beta \tilde e_{n-m}^T, \xi_{n-m}^TC_{2,2} + \alpha \hat e_m^T] =: [\mu_m \tilde e_m^T, \mu_{m-n} \hat e_{m-n}^T].
\eeqo
Thus
\beq \label{eq:row}
 C_{1,1} \xi_m +  \alpha \tilde e_m = \sigma_m \tilde e_m \quad \mbox{or} \quad C_{1,1} \xi_m = (\sigma_m - \alpha) \tilde e_m.
\eeq
Similarly as above we get 
\beq \label{eq:col}
 \xi_m^T C_{1,1} = (\mu_m - \beta) \tilde e_m ^T.
\eeq
So 
\[
 \xi_m^T C_{1,1}\xi_m = (\sigma_m - \alpha) = (\mu_m - \beta).
\]
Hence, since $C_{1,1}  = A^{-1}(1:m)$ is acyclic and nonsingular, Theorem \ref{iff-2}  gives that   $C_{1,1}^{-1}$ is of tree structure. 
Similarly one can prove that $C_{2,2}  = A^{-1}(m+1:n)$
is of tree  structure. 
Next partition $A$ as $C$. So we have
\beqo
 A   =
\left[ \begin{array}{cc}
 A_{1,1} &   A_{1,2}  \\
 A_{2,1}  &  A_{2,2} 
\end {array} \right].
\eeqo
With the Schur-complements
 \beqo
 C \slash C_{2,2} = C_{1,1} - C_{1,2}C^{-1}_{22}C_{2,1},  \quad 
  C \slash C_{1,1} =   C_{2,2} - C_{2,1}C^{-1}_{11}C_{1,2}
   \eeqo
we get 
\beqo
 C^{-1} = 
\left[ \begin{array}{cc}
  (C \slash C_{2,2})^{-1}   &   C^{-1}_{11} C_{1,2} (C \slash C_{1,1})^{-1}  \\
 (C \slash C_{1,1})^{-1}C_{2,1}C^{-1}_{11}    &  (C \slash C_{1,1})^{-1}
\end {array} \right].
\eeqo


But 
\beqo
C \slash C_{2,2} = C_{1,1} - \alpha \tilde e_m \hat e_{m-n}^T C_{2,2}^{-1}\beta \hat e_{m-n} \tilde e_m^T. 
\eeqo
However,  $ \hat e_{m-n}^T C_{2,2}^{-1} \hat e_{m-n}$ is just the first diagonal entry of $C_{2,2}^{-1} =: c_{22}$. 
Thus 
\beqo
C \slash C_{2,2} = C_{1,1} - \alpha \beta c_{22}  \tilde e_m \tilde e_m^T. 
\eeqo
Using the  Sherman-Morrison formula and the fact  that the entries of the last  row and column of $C_{1,1}^{-1}$ are all  the same (see (\ref{eq:row}) and (\ref{eq:col})), we get 
\[
(C \slash C_{2,2})^{-1} = C_{1,1}^{-1} - \gamma_1 \xi_m \xi_m^T,
\]
for some $\gamma_1 \in \bR$. But $C_{1,1}^{-1}$ is of tree structure, so the same holds for $(C \slash C_{2,2})^{-1}$. 
Similarly we can show that
\[
 (C \slash C_{1,1})^{-1} = C_{2,2}^{-1} - \gamma_2 \xi_{m-n} \xi_{m-n}^T,
\]
and thus $  (C \slash C_{1,1})^{-1}$ is of tree structure. 
 Next consider $A_{1,2}$. We obtain with (\ref{eq:row})
 \beqo
 A_{1,2} & =  &  C^{-1}_{11} C_{1,2} (C \slash C_{1,1})^{-1}  C^{-1}_{11} \alpha \tilde e_m \hat e_{m-n}^T ( C_{2,2}^{-1} - \gamma_2 \xi_{m-n} \xi_{m-n}^T) =  \delta_1 \xi_m \xi_{m-n}^T,
 \eeqo
 for some $\delta_1 \in \bR$. 
 
Similarly 
\beqo
 A_{2,1} = \delta_2 \xi_{m-n} \xi_m^T. 
 \eeqo 
 Hence, $A$ is a shifted generalized cyclopes.
  \end{proof}

 
 Using the above facts we get the following Theorem.

  \begin{theorem} \label{theo:cyclopes}
  Let $ A = [a_{ij}] \in \bRn$, with $a_{ij} \neq 0$, be a nonsingular shifted generalized ultrametric  matrix in nested block form 
  \beq \label{eq:block}
A=
\left[ \begin{array}{cc}
A_{1,1} & A_{1,2} \\
A_{2,1} & A_{2,2}
\end {array} \right].
\eeq  
Then  the following  are equivalent.
\begin{enumerate}
 \item
 $A$ is a shifted generalized cyclopes,
 \item 
 $A^{-1}$  is acyclic.
\end{enumerate}
\end{theorem}
\begin{proof}
 {\bf 1. $\rightarrow$ 2.} We can assume that $A$ is in normal block form. Then Theorem  \ref{lem:cyc} gives the result. 
 
{\bf 2.  $\rightarrow $ 1.} 
Since $A^ {-1}$ is acyclic, we have with Theorem \ref{main-structure} 
\beqo
A = T \circ U \circ R = DUF,
\eeqo
where $U$ is of tree structure, $ D = diag(A\hat e_n)$ and $F = diag(\hat e_n^TA)$. So the diagonal entries of $D$ are just the entries of the first column of $A$, while the diagonal entries of $F$ are just the entries of the first row of $A$. Moreover
\beqo
U = D^{-1} AF^{-1}
\eeqo
is symmetric. But  we have with 
 \beqo
 D^{-1}  =:
\left[ \begin{array}{cc}
 D_{11} & 0 \\
 0 & (\omega (A))^{-1}I 
\end {array} \right]  \quad 
F^{-1} =:
\left[ \begin{array}{cc}
 F_{11} & 0 \\
 0 & (\tau (A))^{-1}I 
\end {array} \right]
\eeqo
that 

\beqo
 U =
\left[ \begin{array}{cc}
 D_{11}  A_{1,1}  F_{11} &  D_{11} E_{12}  \\
 E_{21} F_{11}  & (\omega (A))^{-1} (\tau (A))^{-1} A_{2,2}
\end {array} \right].
\eeqo
Thus $A_{2,2}$ must be symmetric. 
Moreover, 
\[
 D_{11} E_{12}  =  (E_{21} F_{11})^T.
\]
But  due to the nested block structure, this leads  to 
\[
 \omega (A_{1,1}) =  \tau (A_{1,1}).
\]
Inductively, we obtain 
\[
 D_{11} = F_{11},
\]
which gives
\[
  A_{1,1}  =  A_{1,1}^T.
\]
Hence, $A$ is a shifted generalized cyclopes. 
\end{proof}

Theorem \ref{theo:cyclopes} generalizes Theorem 3.6 ((i) and (ii)) of \cite{MarMZ03} from \gums \  to shifted \gums. Moreover, note  that  our proof is much shorter.




 
 
 
 
 
    
\section{Examples}
Here  we give some examples which illustrate the results  of the previous section. 

\begin{example} \label{ex1}
{\rm For illustration consider the weighted rooted tree in Figure 1.} 
{\rm 
\noindent
\unitlength0.8cm
\begin{center}
\begin{picture}(12,7)
\put(5.5,0){\circle*{0.2}}
\put(7.5,0){\circle*{0.2}}
\put(3.5,0){\circle*{0.2}}
\put(4.5,2){\circle*{0.2}}
\put(6.5,2){\circle*{0.2}}
\put(5.5,4){\circle*{0.2}}
\put(4,6){\circle*{0.2}}
\put(5.8,0){\makebox(0,0)[lb]{{\bf 5}}}
\put(5.3,1){\makebox(0,0)[lb]{2}}
\put(4.2,0.7){\makebox(0,0)[lb]{1}}
\put(7.3,1){\makebox(0,0)[lb]{1}}
\put(4.3,3){\makebox(0,0)[lb]{3}}
\put(6.3,3){\makebox(0,0)[lb]{2}}
\put(5.0,5){\makebox(0,0)[lb]{1}}
\put(3.8,0){\makebox(0,0)[lb]{{\bf 3}}}
\put(4.7,2){\makebox(0,0)[lb]{{\bf 2}}}
\put(7.8,0){\makebox(0,0)[lb]{{\bf 6}}}
\put(6.8,2){\makebox(0,0)[lb]{{\bf 4}}}
\put(5.8,4){\makebox(0,0)[lb]{{\bf 1}}}
\put(4.4,6){\makebox(0,0)[lb]{{\bf 0}}}
\put(5.5,0){\line(1,2){1}}
\put(7.5,0){\line(-1,2){1}}
\put(3.5,0){\line(1,2){1}}
\put(4.5,2){\line(1,2){1}}
\put(6.5,2){\line(-1,2){1}}
\put(5.5,4){\line(-3,4){1.5}}
\end{picture}
\end{center}
\vspace{.05in}

\begin{center}
{\sc Figure 1} 
\end{center}

}
{\rm The 
$6 \times 6$ matrix of tree 
structure is then given by }
\[
A = \left[ \begin{array}{rrrrrr}
1 & 1 & 1 & 1 & 1 & 1\\
1 & 4 & 4 & 1 & 1 & 1 \\
1 & 4 & 5 & 1 & 1 & 1 \\
1 & 1 & 1 & 3 & 3 & 3 \\
1 & 1 & 1 & 3 & 5 & 3 \\
1 & 1 & 1 & 3 & 3 & 4
\end{array}\right] .
\]
{\rm Here $A$ is an ultrametric matrix, all  the weights are positive. The inverse is given by}

\[
A^{-1} = \frac{1}{6} \left[ \begin{array}{rrrrrr}
11 & -2 & 0 & -3 & 0 & 0\\
-2 & 8 & -6 & 0 & 0 & 0 \\
0 & -6 & 6 & 0 & 0 & 0 \\
-3 & 0 & 0 & 6 & -3 & -6 \\
0 & 0 & 0 & -3 & 3 & 0\\
0 & 0 & 0 & -6 & 0 & 6
\end{array}\right] .
\]
{\rm Thus, clearly as shown in Theorem \ref{invM}, matrix $A$ is an inverse acyclic $M$-matrix. The graph of $A^{-1}$ is just the graph as in Figure 1.}
\end{example}

\begin{example} 
{\rm Next  consider   again the tree as given in Figure 1 but now with different weights.}

{\rm 
\noindent
\unitlength0.8cm
\begin{center}
\begin{picture}(12,6.5)
\put(5.5,0){\circle*{0.2}}
\put(7.5,0){\circle*{0.2}}
\put(3.5,0){\circle*{0.2}}
\put(4.5,2){\circle*{0.2}}
\put(6.5,2){\circle*{0.2}}
\put(5.5,4){\circle*{0.2}}
\put(4,6){\circle*{0.2}}
\put(5.8,0){\makebox(0,0)[lb]{{\bf 5}}}
\put(5.3,1){\makebox(0,0)[lb]{2}}
\put(4.2,0.7){\makebox(0,0)[lb]{2}}
\put(7.3,1){\makebox(0,0)[lb]{1}}
\put(4.3,3){\makebox(0,0)[lb]{2}}
\put(6.3,3){\makebox(0,0)[lb]{3}}
\put(5.0,5){\makebox(0,0)[lb]{-6}}
\put(3.8,0){\makebox(0,0)[lb]{{\bf 3}}}
\put(4.7,2){\makebox(0,0)[lb]{{\bf 2}}}
\put(7.8,0){\makebox(0,0)[lb]{{\bf 6}}}
\put(6.8,2){\makebox(0,0)[lb]{{\bf 4}}}
\put(5.8,4){\makebox(0,0)[lb]{{\bf 1}}}
\put(4.4,6){\makebox(0,0)[lb]{{\bf 0}}}
\put(5.5,0){\line(1,2){1}}
\put(7.5,0){\line(-1,2){1}}
\put(3.5,0){\line(1,2){1}}
\put(4.5,2){\line(1,2){1}}
\put(6.5,2){\line(-1,2){1}}
\put(5.5,4){\line(-3,4){1.5}}
\end{picture}
\end{center}

\vspace{.05in}

\begin{center}
{\sc Figure 2} 
\end{center}
}
{\rm  With the weights  as in Theorem \ref{invN0} we have}
\[
A = \left[ \begin{array}{rrrrrr}
-6 & -6 & -6 & -6 & -6 & -6\\
-6 & -4 & -4 & -6 & -6 & -6 \\
-6 & -4 & -2 & -6 & - 6& -6 \\
-6 & -6 & -6 & -3 & -3 & -3 \\
-6 & -6 & -6 & -3 & -1 & -3 \\
-6 & -6 & -6 & -3 & -3 & -2
\end{array}\right] .
\]

{\rm The inverse is then}
\[
A^{-1} = \frac{1}{6} \left[ \begin{array}{rrrrrr}
4 & -3 & 0 & -2 & 0 & 0\\
-3 & 6 & -3 & 0 & 0 & 0 \\
0 & -3 & 3 & 0 & 0 & 0 \\
-2 & 0 & 0 & 11 & -3 & -6 \\
0 & 0 & 0 & -3 & 3 & 0\\
0 & 0 & 0 & -6 & 0 & 6
\end{array}\right] .
\]
{\rm Thus, with Theorem \ref{invN0} matrix $A$ is an inverse acyclic $N_0$-matrix. The graph of $A^{-1}$ is just the graph as in Figure 2.}
\end{example}

\begin{example} 
{\rm Next  we illustrate Theorem \ref{invN0}. Consider the trees as given in Figure 3.}

{\rm 
\noindent
\unitlength0.8cm
\begin{picture}(18,7)
\put(7.5,0){\circle*{0.2}}
\put(4.5,2){\circle*{0.2}}
\put(6.5,2){\circle*{0.2}}
\put(5.5,4){\circle*{0.2}}
\put(4,6){\circle*{0.2}}


\put(7.3,1){\makebox(0,0)[lb]{2}}
\put(4.3,3){\makebox(0,0)[lb]{4}}
\put(6.3,3){\makebox(0,0)[lb]{1}}
\put(5.0,5){\makebox(0,0)[lb]{-5}}
\put(4.7,2){\makebox(0,0)[lb]{{\bf 2}}}
\put(7.8,0){\makebox(0,0)[lb]{{\bf 4}}}
\put(6.8,2){\makebox(0,0)[lb]{{\bf 3}}}
\put(5.8,4.2){\makebox(0,0)[lb]{{\bf 1}}}
\put(4.4,6){\makebox(0,0)[lb]{{\bf 0}}}
\put(7.5,0){\line(-1,2){1}}
\put(4.5,2){\line(1,2){1}}
\put(6.5,2){\line(-1,2){1}}
\put(5.5,4){\line(-3,4){1.5}}


\put(14.5,0){\circle*{0.2}}
\put(10.5,0){\circle*{0.2}}
\put(11.5,2){\circle*{0.2}}
\put(13.5,2){\circle*{0.2}}
\put(12.5,4){\circle*{0.2}}
\put(11,6){\circle*{0.2}}



\put(11.2,0.7){\makebox(0,0)[lb]{1}}
\put(14.3,1){\makebox(0,0)[lb]{2}}
\put(11.3,3){\makebox(0,0)[lb]{4}}
\put(13.3,3){\makebox(0,0)[lb]{1}}
\put(12.0,5){\makebox(0,0)[lb]{2}}
\put(10.8,0){\makebox(0,0)[lb]{{\bf 3}}}
\put(11.7,2){\makebox(0,0)[lb]{{\bf 2}}}
\put(14.8,0){\makebox(0,0)[lb]{{\bf 5}}}
\put(13.8,2){\makebox(0,0)[lb]{{\bf 4}}}
\put(12.8,4.2){\makebox(0,0)[lb]{{\bf 1}}}
\put(11.4,6){\makebox(0,0)[lb]{{\bf 0}}}


\put(14.5,0){\line(-1,2){1}}

\put(10.5,0){\line(1,2){1}}
\put(11.5,2){\line(1,2){1}}
\put(13.5,2){\line(-1,2){1}}
\put(12.5,4){\line(-3,4){1.5}}
\end{picture}
\vspace{.05in}

\begin{center}
{\sc Figure 3} 
\end{center}
}
{\rm On the left in Figure 3 we obtain a $4 \times 4$ inverse $N_0$-matrix. On the right we get a $5 \times 5$ inverse $M$-matrix. Combining these trees we get 
the tree given in Figure 4.}

{\rm 
\noindent
\unitlength0.8cm
\begin{picture}(18,7)
\put(7.5,0){\circle*{0.2}}
\put(4.5,2){\circle*{0.2}}
\put(6.5,2){\circle*{0.2}}
\put(5.5,4){\circle*{0.2}}
\put(4,6){\circle*{0.2}}

\put(5.5,4){\line(1,0){7}}

\put(7.3,1){\makebox(0,0)[lb]{2}}
\put(4.3,3){\makebox(0,0)[lb]{4}}
\put(6.3,3){\makebox(0,0)[lb]{1}}
\put(5.0,5){\makebox(0,0)[lb]{-5}}
\put(4.7,2){\makebox(0,0)[lb]{{\bf 2}}}
\put(7.8,0){\makebox(0,0)[lb]{{\bf 4}}}
\put(6.8,2){\makebox(0,0)[lb]{{\bf 3}}}
\put(5.8,4.2){\makebox(0,0)[lb]{{\bf 1}}}
\put(4.4,6){\makebox(0,0)[lb]{{\bf 0}}}
\put(7.5,0){\line(-1,2){1}}
\put(4.5,2){\line(1,2){1}}
\put(6.5,2){\line(-1,2){1}}
\put(5.5,4){\line(-3,4){1.5}}


\put(14.5,0){\circle*{0.2}}
\put(10.5,0){\circle*{0.2}}
\put(11.5,2){\circle*{0.2}}
\put(13.5,2){\circle*{0.2}}
\put(12.5,4){\circle*{0.2}}




\put(11.2,0.7){\makebox(0,0)[lb]{1}}
\put(14.3,1){\makebox(0,0)[lb]{2}}
\put(11.3,3){\makebox(0,0)[lb]{4}}
\put(13.3,3){\makebox(0,0)[lb]{1}}
\put(10.8,0){\makebox(0,0)[lb]{{\bf 7}}}

\put(11.7,2){\makebox(0,0)[lb]{{\bf 6}}}

\put(14.8,0){\makebox(0,0)[lb]{{\bf 9}}}
\put(13.8,2){\makebox(0,0)[lb]{{\bf 8}}}
\put(12.8,4.2){\makebox(0,0)[lb]{{\bf 5}}}
\put(8.0,4.2){\makebox(0,0)[lb]{{2-(-5)= 7 }}}


\put(14.5,0){\line(-1,2){1}}

\put(10.5,0){\line(1,2){1}}
\put(11.5,2){\line(1,2){1}}
\put(13.5,2){\line(-1,2){1}}
\end{picture}
\vspace{.05in}

\begin{center}
{\sc Figure 4} 
\end{center}
}

{\rm  We then obtain a matrix which is of tree structure}
\[
A = \left[ \begin{array}{rrrrrrrrr}
-5 & -5 & -5 & -5      & -5 & -5 & -5 & -5 & -5\\
-5 & -1 & -5 & -5      & -5 & -5 & -5 & -5 & -5\\
-5 & -5 & -4 & -4      & -5 & -5 & -5 & -5 & -5\\     
-5 & -5 & -4 & -2      & -5 & -5 & -5 & -5 & -5\\ 
-5 & -5 & -5 & -5 &   2 & 2 & 2 & 2 & 2 \\
-5 & -5 & -5 & -5 &   2 & 6 & 6 & 2 & 2 \\
-5 & -5 & -5 & -5 &   2 & 6 & 7 & 2 & 2 \\
-5 & -5 & -5 & -5 &   2 & 2 & 2 & 3 & 3 \\
-5 & -5 & -5 & -5 &   2 & 2 & 2 & 3 & 5 
\end{array}\right] .
\]
{\rm The inverse is given by}
\[
A^{-1} = \frac{1}{140}\left[ \begin{array}{rrrrrrrrr}
167 & -35 & -140 & 0 & -20      & 0 & 0 & 0 & 0\\
-35 & 35 & 0 & 0      & 0 & 0 & 0 & 0 & 0\\
-140 & 0 & 210 & -70 & 0     & 0 & 0 & 0 & 0\\ 
0 & 0 & -70 & 70     & 0 & 0 & 0 & 0 & 0\\
-20 & 0 & 0 & 0 &   195 & -35 & 0 & -140 & 0 \\
0 & 0 & 0 & 0 &   -35 & 175 & -140 & 0 & 0 \\
0 & 0 & 0 & 0 &   0 & -140 & 140 & 0 & 0 \\
0 & 0 & 0 & 0 &   -140 & 0 & 0 & 210 & -70 \\
0 & 0 & 0 & 0 &   0 & 0 & 0 & -70 & 70 
\end{array}\right] .
\]
{\rm Here we then have $A^{-1} \in L_3$ and its graph is the same as in Figure 4.}

\end{example}
The next example gives a shifted generalized cyclopes.
\begin{example}
 {\rm Let $A$ be given by}
\[
A= \left[ \begin{array}{cccccc}
-1 & -2 & -3 & -1 &-1 & -1 \\ 
-2  & -2 & -3 & -1 & -1 & -1 \\ 
-3  & -3 & -3 & -1 & -1 & -1 \\
-5 & -5 & -5 & 1 & 1 &  1 \\
-5 & -5 & -5 & 1 & 3 & 3 \\
-5 & -5 & -5 & 1 & 3 & 5
\end{array}\right] .
\]
{\rm $A$ is a shifted generalized cyclopes with eye 3. The inverse is given by }
\[
A^{-1} = \frac{1}{8}\left[ \begin{array}{cccccc}
8 & -8 & 0 & 0 & 0 & 0  \\ 
-8 & 16 & -8 & 0 & 0 & 0  \\
0 & -8 & 7 & -1 & 0 & 0 \\ 
 0 & 0 & -5 & 7 & -4 & 0 \\
 0  &  0 &  0 & -4 & 8 & -4 \\
  0&  0 &  0&  0 & -4 & 4
  \end{array}\right] .
\]
  
\end{example}

\bibliographystyle{abbrv}
\bibliography{acyclic.bib}

\end{document}